\documentclass{birkjour}
\usepackage{xcolor}
\RequirePackage[all]{xy}
\usepackage[colorlinks=true, allcolors=blue]{hyperref}
\usepackage{amssymb}
 \usepackage{amsthm}
\newtheorem{thm}{Theorem}[section]

\newtheorem{lem}[thm]{Lemma}
\newtheorem{prop}[thm]{Proposition}
\theoremstyle{definition}
\newtheorem{defn}[thm]{Definition}
\theoremstyle{remark}
\newtheorem{rem}[thm]{Remark}

\newtheorem*{ex}{Example}
\numberwithin{equation}{section}

\newcommand{\BibTeX}{B\kern-0.1emi\kern-0.017emb\kern-0.15em\TeX}
\newcommand{\XYpic}{$\mathrm{X\kern-0.3em\raisebox{-0.18em}{Y}}$-$\mathrm{pic}\,$}

\newcommand{\cl}{C \kern -0.1em \ell}  

\newcommand{\ed}{\end{document}}
\newcommand{\rev}[1]{\textcolor{red}{#1}}
\begin{document}
\title[Fractional  Sobolev Spaces  with   Kernel  Function ]
 {Fractional Sobolev Spaces  with   Kernel  \\Function on  Compact Riemannian \\ Manifolds}
\author[A. Aberqi]{Ahmed Aberqi}
\address{%
Laboratory LAMA, National School of Applied Sciences\\
Sidi Mohamed Ben Abdellah University\\
BP 1796, 30000 Fez\\
 Morocco.}
\email{ahmed.aberqi@usmba.ac.ma}
\thanks{Corresponding author: D.D. Repov\v{s}}
\author[A. Ouaziz]{Abdesslam Ouaziz}
\address{%
Laboratory LAMA, Faculty of Sciences Dhar Elmahraz\\
Sidi Mohamed Ben Abdellah University\\
BP 1796, 30000 Fez\\
 Morocco.}
\email{abdesslam.ouaziz@usmba.ac.ma}
\author[D.D. Repov\v{s}]{Du\v{s}an D. Repov\v{s}}
\address{%
Faculty of Education and Faculty of Mathematics and Physics, \\
University of Ljubljana \\
 \&
Institute of Mathematics, Physics and Mechanics, \\ 1000 Ljubljana \\ Slovenia.}
\email{dusan.repovs@guest.arnes.si}
\date{}
\begin{abstract}
In this paper, a new class of Sobolev spaces with kernel function satisfying a Lévy-integrability type condition on compact Riemannian manifolds is presented. We establish the properties of separability, reflexivity, and completeness. An embedding result is also proved. As an application, we prove the existence of solutions for a nonlocal elliptic problem involving the fractional $p(\cdot, \cdot)$-Laplacian operator. As one of the main tools,  topological degree theory is applied. 
\end{abstract}
\label{page:firstblob}
\maketitle
\section{Introduction}\label{s1} 
Let $(\mathcal{M}, {g})$ be a compact Riemannian manifold of dimension N. The purpose of this paper is to present fundamental properties of a new class of Sobolev spaces with general kernel on $(\mathcal{M}, {g})$. Additionally, we shall solve the following equation
 \begin{eqnarray}\label{k1}
\begin{gathered}
\left\{\begin{array}{llll}
(\mathcal{L}_{{g}}  ^{K})_{p(y, \cdot)} w(y)&=\lambda   \beta(y)|w(y)|^{r(y)-2}w(y)+  f(y, w(y)) & \text { in } & \mathcal{U}, \\ 
\hspace{2cm}  w&=0 & \text { in }&  \mathcal{M} \backslash \mathcal{U}.
\end{array}\right.
\end{gathered}
\end{eqnarray}
Here, $s\in(0,1)$  is fixed,   $r\in C( \mathcal{U},$  $(1, \infty))$, $\lambda>0 $, $\mathcal{U} \subset  \mathcal{M}$  is an open bounded  subset  of $\mathcal{M}$, $\beta $ is a suitable potential  function in $\mathbb{R}^{+}$  with $\beta \in L^{\infty}(\mathcal{U})$,     $p\in \mathcal{C}( \mathcal{U}  \times  \mathcal{U}, (1, \infty))$ satisfies the following conditions

 \begin{equation}\label{l3}
 p(z,{a})=p({a}, z), \  \text { for every } (z,{a}) \in  \mathcal{M}^{2},
 \end{equation} 
 \begin{equation}\label{l30}
 1< p^{-}=\min_{(y, z)\in\mathcal{M}^{ 2}} p(y, z)\leq p(y, z)<p^{+} = \sup_{(y, z)\in\mathcal{M} ^{2}} p(y, z),
 \end{equation} 
and  $f:\mathcal{M} \times \mathbb{R} \to  \mathbb{R}$ is a Carathéodory function such that:

 $(\mathcal{B}_{1})\label{b1}$ There exist $\alpha>0$ and a continuous function $q:\mathcal{M}  \to  (1, +\infty)$
 such that
 $$1< q(y)< p^{\star}_{s}(y)=\frac{N p(y,y)}{N- sp(y,y)}$$
and
$$f(y,  z) \leq  \alpha \left( 1+ \vert z\vert^{q({y})-1}\right), ~  \text { a.e.  } ~  y \in  \mathcal{M}, ~  z\in \mathbb{R},$$
$$r^{+} = \sup_{y\in\mathcal{M}} r(y)
\leq  q^{-}=\min _{y\in\mathcal{M}}q(y)\leq  p^{\star}_{s}(y)$$  
  and the operator   $(\mathcal{L}_{{g}}  ^{K})_{p(y,\cdot)}$ is defined by 
  \begingroup\makeatletter\def\f@size{8}\check@mathfonts   
 $$ (\mathcal{L}_{{g}}^{K})_{p(y,\cdot)} w(y)=2 \lim _{\varepsilon \rightarrow 0^{+}} \int_{\mathcal{M} \backslash \mathfrak{B}_{\varepsilon}(y)}|w(y)-w(z)|^{p(y, z)-2}(w(y)-w(z)) K(y, z) d v_{{g}}(z),$$ 
 \endgroup
  for every  $z \in \mathcal{M},$ where, 
$$\mathfrak{B}_{\varepsilon}(y)= \{z \in \mathcal{M}: d_{{g}}(y, z)< \varepsilon \}
~~
\hbox{and}
~~
 d v_{{g}}(z)=|{g}_{ij}|^{\frac{1}{2}} dz.$$ (see Definition~\ref{d1}).  
 Furthermore, $K: \mathcal{U} \times \mathcal{U} \to (0, +\infty) $ is a symmetric  kernel function satisfying the following variant of  Lévy-integrability type condition
\begin{equation}\label{equation50}
g  K\in L^{1}\left( \mathcal{M}\times \mathcal{M}, dv_{g}(y) dv_{g}(z)\right),
~~
\hbox{where}
~~
g(y, z) =\min\{ d_{{g}}(y, z), 1\}  
\end{equation}   
and the following coercivity condition for some $\alpha_{0}>0$
  \begin{equation}\label{equation51} \frac{\alpha_{0}}{d_{{g}}(y, z)^{N+s p(y, z)}} \leq K(y, z),  \  \text { a.e. } \,  (y, z)\in  \mathcal{M} ^{2},\,\,     y \neq z. 
 \end{equation}
We give examples of symmetric kernel functions that satisfy Lévy-integrability and coercivity type conditions.
\begin{ex}
The following functions satisfy   conditions \eqref{equation50}-\eqref{equation51}.
\begin{enumerate}
\item[$\clubsuit$]  $K(y, z)= d_{g}(y,z)^{-N-sp(y,z))}.$
\item[$\clubsuit$]   $K(y, z)= \frac{\alpha_{0}}{d_{g}(y,z)^{N+sp(y,z))}},$ where $\alpha_{0}$ is a positive real.
\item[$\clubsuit$]   $K(y, z)= \exp \left( \frac{1}{d_{g}(y,z)^{N+sp(y,z))}} \right). $
\item[$\clubsuit$]   $K(y, z)= \exp \left(  -\delta d_{g}(y,z)^{2}\right),  $  where $\delta $ is a positive real.
\end{enumerate}  
\end{ex}
Recently, results on fractional Sobolev spaces and problems involving the $p(y,\cdot)$-operator and their applications have received a lot of attention. For example, Kaufmann, Rossi, and Vidal \cite{kaufmann} first introduced the new class $W^{s,  q(y), p(y, z)}(\mathcal{U})$ defined by
$$W^{s,  q(y), p(y, z)}(\mathcal{U})=\{ w\in L^{q(y)} (\mathcal{U}):  \int_{\mathcal{M} \times\mathcal{M}} \frac{|w(y)-w(z)|^{p(y, z)}}{K(y, z) }dydz< +\infty \},$$
where $q\in  C(\overline{\mathcal{U}}, (1, \infty))$ and   $K(y, z)=|y-z|^{N+s p(y, z)},$ and they proved the existence of a compact embedding 
$$\displaystyle W^{s,  q(y, p(y, z)}(\mathcal{U})\hookrightarrow   L^{r(y)} (\mathcal{U}),
~\hbox{ for every}~
r\in C(\mathcal{U})$$
such that
$1< r(y)<p^{\star}_{s} (y),$
for every $y\in \overline{\mathcal{U}}.$ They
also  studied solvability of  the following fractional $  p(y,\cdot)$-Laplacian problem
\begin{eqnarray}
\begin{gathered}
\left\{\begin{array}{llll}
\mathcal{L} w(y)+|w(y)|^{q(y)-2} w(y)&=h(y) & \text { in }& \mathcal{Q}, \\ 
   \hspace{2cm}w&=0 & \text { in }& \partial \mathcal{Q},
\end{array}\right.
\end{gathered}
\end{eqnarray}
with $h \in L^{a(y)}(\mathcal{Q}), a(y)>1$. 

For more results on the functional framework, we refer to Bahrouni and R\u{a}dulescu~\cite{bahrouni} who   proved  the solvability of the following problems
\begin{eqnarray}
\begin{gathered}
\left\{\begin{array}{llll}
\mathcal{L} w(y)+|w(y)|^{q(y)-2}w(y)&=\lambda|w(y)|^{r(y)-2}w(y) & \text { in }& \mathcal{U}, \\ 
\hspace{2cm}  w&=0 &\text { in }& \mathbb{R}^{N}\backslash \mathcal{U},
\end{array}\right.
\end{gathered}
\end{eqnarray}
by using Ekeland's variational method, where  $\mathcal{U}$ is  an open bounded  subset of   $\mathbb{R}^{N}, \lambda>0, $  
$\displaystyle r(y)< p^{-}= \min_{(y, z)\in \mathcal{U}\times \mathcal{U}} p(y, z), $ and  $\mathcal{L} w$ is the fractional $p(y,\cdot)$ Laplacian operator. 

Bahrouni \cite{bahroun} continued to study  the space $W^{s,  q(y, p(y, z)}(\mathcal{U}).$ More specially, he  proved  the strong  comparison principle for $\left(-\Delta\right)^{s}_{p(y, .)}$ and by using the sub-supersolution method, he showed the solvability of the following nonlocal equation
\begin{eqnarray}
\begin{gathered}
\left\{\begin{array}{llll}
\left(-\Delta\right)^{s}_{p(y)} w(\cdot)&=h(y, w(\cdot)) & \text { in }& \mathcal{U}, \\ 
 \hspace{2cm}  w&=0 &\text { in }& \mathbb{R}^{n}\backslash \mathcal{U},
\end{array}\right.
\end{gathered}
\end{eqnarray}
where $\mathcal{U}$ is an open bounded domain, $s\in(0,1),$  $p $ is a continuous function, and $h$ satisfies the following growth
$$
|h(y, z)| \leq A_{1}\vert z\vert^{r(y)-1}+A_{2},~ \text{ for every } (y, z) \in \mathbb{R}^{N+1}, ~
$$$ \hbox{where}~
r\in C( \mathbb{R}^{N},\mathbb{R}),~ 1< r(y)< p^{\star}_{s} (y),~\hbox{for every}~ y\in \mathbb{R}^{N}.$

 The generalized fractional Sobolev space was studied in  \cite{ bahroun, bahrouni, kaufmann} and further developed in  \cite{ho}.  They proved a fundamental compact embedding for this space and investigated the multiplicity and boundedness   of solutions to the following problem
\begin{eqnarray}
\begin{gathered}
\left\{\begin{array}{llll}
\left(-\Delta\right)^{s} _{p(y)} w(\cdot)&=f(\cdot, w(\cdot)) & \text { in }& \mathcal{U}, \\ 
 \hspace{2cm}  w&=0 &\text { in }& \mathbb{R}^{N}\backslash \mathcal{U},
\end{array}\right.
\end{gathered}
\end{eqnarray}
where $p\in C(\mathbb{R}^{N}\times\mathbb{R}^{N}, (1, +\infty))$ is
such that  $p$ satisfies the following conditions
\begin{equation}\label{l20O}
p(y, z)=p( z, y),  \  \text { for every } (y, z) \in  \mathbb{R}^{2N}, 
 \end{equation}
\begin{equation}\label{l700}
1<\inf_{(y, z) \in \mathbb{R}^{N}\times\mathbb{R}^{N}}p(y, z)\leq  p(y, z)< \sup_{(y, z) \in \mathbb{R}^{N}\times\mathbb{R}^{N}}p(y, z)< \frac{N}{s},
 \end{equation}
   $f:\mathbb{R}^{N}\times  \mathbb{R} \rightarrow \mathbb{R} $ is a Carathéodory function,  and  $\left(-\Delta_{p(y)}\right)^{s}$ is an operator defined by $$\left(-\Delta\right)^{s} _{p(y)} w(y)= 2 \lim _{\varepsilon \rightarrow 0^{+}} \int_{ \mathbb{R}^{N} \backslash B_{\varepsilon}(y)} \frac{|w(y)-w(z)|^{p(y, z)-2}(w(y)-w(z))}{|y-z |^{N+s  p(y,z)}}  d z,$$
   $ ~\hbox{ where}~  B_{\varepsilon}(y)=\{ z \in \mathbb{R}^{N}: |z-y|< \varepsilon \}.$
   
   The approaches for ensuring the existence of weak solutions for a class of nonlocal fractional problems with variable exponents were addressed in greater depth in \cite{ bennouna, aberqi, ayazoglu, bahroun, bahrouni, benslimane, berkovits, chang,  chen,  ho, kaufmann,  liu, radulescu, servadei} and the references therein.
  In the non-Euclidean case, classical Sobolev spaces on  Riemannian manifolds have been investigated for more than seventy years  \cite{aubin,  hebey, palais}. The theory of these spaces has been applied to
   isoperimetrical inequalities \cite{hebey} and the Yamabe problem \cite{trudinger}. In \cite{gaczkowski}  the authors investigated the theory of generalized Sobolev spaces on compact Riemannian manifolds. Moreover, they proved the compact embeddability of these spaces into the H\"older space. They also studied a PDE problem involving
   $p(\cdot)$-Laplacian operator.   
  
  In addition, the authors in \cite{gorka} studied variable exponent function spaces on complete non-compact Riemannian manifolds. They used classical assumptions on the geometry to establish compact embeddings between Sobolev spaces and the H\"older function space. 
   Finally, they also showed the existence of solutions to the $p(\cdot)$-Laplacian problem.
The authors in \cite{guo} introduced the fractional Sobolev spaces on  Riemannian manifolds. As a consequence, they investigated fundamental properties, such as compact embeddings, completeness, density, separability, and reflexivity. They also investigated the existence of solutions to the following equation:
\begin{eqnarray}
\begin{gathered}
\left\{\begin{array}{llll}
\left(-\mathcal{L}_{{g}}\right)_p^{s} w(y)&=f(y, w(y)) & \text { in }& \mathcal{U}, \\ 
 \hspace{1.6cm}w&=0 & \text { in }& \mathcal{M}  \backslash \mathcal{U},
\end{array}\right.
\end{gathered}
\end{eqnarray}
where $\mathcal{M} $ is a compact  manifold of dimension d,  $ \mathcal{U} \subset \mathcal{M}$ is an open bounded subset of $\mathcal{M}$,  $s \in(0,1)$,  $p>1$ with  $d>ps $,  $f:\mathcal{M}\times\mathbb{R} \rightarrow \mathbb{R}$ is a Carathéodory function, and the operator $\left(-\mathcal{L}_{g}\right)_{p}^{s} w(y),$  $ y\in \mathcal{M}$ is defined by
$$
\left(-\mathcal{L}_{{g}}\right)_{p}^{s} w(y)=2 \lim _{\varepsilon \rightarrow 0^{+}} \int_{\mathcal{M} \backslash \mathfrak{B}_{\varepsilon}(y)} \frac{|w( y)-w(z)|^{p-2}(w(y)-w(z))}{\left(d_{{g}}(y, z)\right)^{d+p s}} d v_{{g}}(z).
$$
 
Aberqi, Benslimane, Ouaziz, and Repov\v{s} \cite{aberqi} introduced  the space \\$W^{s, p(y, z)}(\mathcal{M})$ and proved some important properties of this space and studied the following problem
$$
(\mathcal{P})\left\{\begin{array}{l}
\left(-\mathcal{L}_{{g}}\right)_{p(y,.)}^{s} w(y)+\mathcal{V}(y)|w(y)|^{q(y)-2} w=h(y, w(y)) \quad \text { in } \mathcal{U}\\

\left.w\right|_{\partial \mathcal{U}}=0.
\end{array}\right.$$

Fractional Sobolev spaces and problems involving the $p(\cdot, \cdot)$-Laplacian operator have attracted significant attention in recent decades. This class of operators appears rather naturally in a variety of applications, including optimization  and financial mathematics, we cite the well-known example by Carbotti, Dipierro, and Valdinoci \cite{carb19} who obtained the following equation:
$$
\frac{\partial V}{\partial t}\left(S_t, t\right)+\mathcal{A} V\left(S_t, t\right)=r V\left(S_t, t\right)-r \frac{\partial V}{\partial t}(0, t) S_t,
$$
where $\mathcal{A}:=a \partial^2-b(-\Delta)^s$ with
 $a, b \geqslant 0$ and $r \in \mathbb{R}$. Here, $S_t$ is the price at time $t$ and $V$ the value of option.
 They are also useful in optimal control, engineering, quantum mechanics, obstacle problems, elasticity, image processing, minimal surfaces, stabilization of Lévy processes, game theory, population dynamics, fluid filtration, and stochastics, see for example \cite{applebaum, bae, caffarelli, chang, choi, diening, gilboa, ruzicka, servadei} and the references therein.

Our work's novelty is extending general Sobolev spaces to  Sobolev spaces
$W_{K}^{q(y), p(y,z)}(\mathcal{U})$ 
with   kernel  function
$K$ 
on $\mathcal{M}.$  We shall prove important properties of this new class of spaces. In particular,  we shall investigate the existence of solutions to problem \eqref{k1} using the topological degree method. This work generalizes previous results \rev{\cite{bennouna, aberqi, bahroun, bahrouni, biswas,  guo, ho,  kaufmann}}. However, the main difficulty is presented by the fact that the  $p(\cdot, \cdot)$-Laplacian operator has a more complicated nonlinearity than the $p$-Laplacian operator. For example, it is non-homogeneous. Other complications are due to the non-Euclidean framework of our problem. Also  checking for example the density of the space $C^{\infty}(\mathcal{M})$ in  $W_{K}^{q(y), p(y,z)}(\mathcal{U}),$ because the notion the translation in Riemannian manifolds is not defined.
To the best of our knowledge, there were no such results prior to this work.

Our first major result is the following theorem.
\begin{thm}\label{embedding}
 Suppose that $(\mathcal{M} , {g})$  is  a compact $N$-dimensional Riemannian manifold,   
 $\mathcal{U}$  is  a smooth open subset of $\mathcal{M},$ 
 $K:\mathcal{U}\times\mathcal{U} \rightarrow (0, +\infty) $ is a  symmetric function satisfying  Lévy-integrability and coercivity  conditions, $q\in C^{+}(\mathcal{M}), $ 
  $p:\mathcal{U}\times\mathcal{U} \rightarrow (1, +\infty)$ satisfies conditions  \eqref{l3}-\eqref{l30} and
   $$sp(y,z)<N, ~ p(y, y)< q(y),~\hbox{ for every} ~   (y, z)\in\mathcal{U}^{2},$$
and $\mathfrak{\ell}:\overline{\mathcal{M}}  \rightarrow (1,+ \infty)$  is a continuous variable exponent such that $$p^{*}_{s}(y)= \frac{Np(y,y)}{N-sp(y,y)}>\mathfrak{\ell}(y)\geq\mathfrak{\ell}^{-}=\min_{y \in  \overline{\mathcal{M}}}\mathfrak{\ell}(y)>1.$$
  Then the space  $W_{K}^{q(y), p(y,z)}(\mathcal{U})$ is continuously embeddable in $L ^{\mathfrak{\ell}(y)}(\mathcal{U})$ and there exists a positive constant $C= C(N,  s, p, q, \mathcal{U} )$ such that  $$|w|_{L ^{{l}(y)}(\mathcal{U})}\leq  ||w||_{W_{K}^{q(y), p(y,z)}(\mathcal{U})},\ \text{for every}\quad w\in W_{K}^{q(y), p(y,z)}(\mathcal{U}).$$  Moreover, this embedding is compact.
\end{thm}
Our second main result is related to the investigation of the following fractional $p(y,\cdot)$-Laplacian  problem with a general kernel $K$
\begin{eqnarray*}
\begin{gathered}
(\ref{k1})\left\{\begin{array}{llll}
(\mathcal{L}_{{g}}  ^{K})_{p(y,\cdot)} w(y)&=\lambda   \beta(y)|w(y)|^{r(y)-2}w(y)+ f(y, w(y)) & \text { in }& \mathcal{U}, \\ 
\hspace{2cm}  w&=0 &\text { in }&  \mathcal{M} \backslash \mathcal{U}.
\end{array}\right.
\end{gathered}
\end{eqnarray*}
Using Berkovits' topological degree, we study the existence of solutions and prove the following theorem. 
\begin{thm}\label{existence}
Suppose that $(\mathcal{M} , {g})$  is  a compact $N$-dimensional Riemannian manifold,   
 $\mathcal{U}$  is  a smooth open subset of $\mathcal{M},$ 
 $K:\mathcal{U}\times\mathcal{U} \rightarrow (0, +\infty) $ is a  symmetric function satisfying  Lévy-integrability and coercivity  conditions. Assume that assumption $ (\mathcal{B}_{1})$  holds. Then problem \eqref{k1} has at least one weak solution $w\in W_{K}^{q(y), p(y,z)}(\mathcal{U}). $
\end{thm}
The  paper is organized as follows:  
In Section~\ref{rappel}, we collect the main definitions and properties of generalized Lebesgue spaces and generalized  Sobolev spaces on compact manifolds and provide crucial background on recent Berkovits degree theory. 
In Section~\ref{framework}, we establish completeness, separability, and reflexivity properties of our spaces (Lemmas~\ref{completness}, ~\ref{unif-convex}, and~\ref{separability}).
In Section~\ref{1}, we prove our first main result (Theorem~\ref{embedding}). 
In Section~\ref{2} we prove our second main result (Theorem~\ref{existence}).
Finally, in Section~\ref{appendix}, we  prove some lemmas needed for the proofs of our main results.
\section{Preliminaries}\label{rappel}
\subsection{ Generalized  Lebesgue Spaces  on Compact  Manifolds}
Throughout this section,  $(\mathcal{M} , {g})$ will be a compact Riemannian manifold of dimension $\displaystyle N$.
 To start,   we briefly review some fundamental Riemannian geometry concepts that will be needed. For more details see  \cite{aubin,  guo, hebey}.  
 
 A local chart on $\mathcal{M}$  is a pair $(\mathcal{U}, \varphi)$, where $\mathcal{U}$ is an open subset of $\mathcal{M}$ and $\varphi$ is a homeomorphism of  $\mathcal{U}$ onto an open subset of $\mathbb{R}^{N}.$ Furthermore, a  collection $(\mathcal{U}_{i}, \varphi_{i})_{i \in J}$ of local charts  such that $\mathcal{M}= \bigcup_{j\in J} \mathcal{U}_{j},$  is called an atlas  of  manifold $\mathcal{M}$.   For some atlas  $(\mathcal{U}_{j}, \varphi_{j})_{j \in J}$of $\mathcal{M},$ we  say that  a family $(\mathcal{U}_{j}, \varphi_{j}, \beta _{j})_{j\in J}$ is 
 a partition of unity subordinate to  the covering $(\mathcal{U}_{j}, \varphi_{j})_{j\in J}$ if the following holds:
\begin{itemize}
\item[1)] $\sum_{j\in J}  \beta_{j}=1, $
\item[2)] $ (\mathcal{U}_{j}, \varphi_{j})_{j\in J}$ is an atlas of  $\mathcal{M},$ 
\item[3)]      $ \text{supp}    (\beta_{j}) \subset  \mathcal{U}_{j},$  for every $j\in J.$  
\end{itemize}

 \begin{defn}(see \cite{hebey})\label{d1}
  Suppose that $\displaystyle w:\mathcal{M}\rightarrow\mathbb{R}$ is  a continuous function with compact support, $(\mathcal{U}_{j}, \varphi_{j})_{j\in J}$ is an atlas of  $\mathcal{M},$ and  $(\mathcal{U}_{j}, \varphi_{j}, \beta _{j})_{j\in J}$ is a partition of unity subordinate to $\displaystyle (\mathcal{U}_{j}, \varphi_{j})_{j\in J}.$ We  define the Riemannian measure of $\displaystyle w$  in $\mathcal{M}$ as follows:
  $$\int_{\mathcal{\mathcal{M}}} w(y)d v_{{g}}(y)= \sum_{j\in J} \int_{       \varphi_{j} (\mathcal{U}_{j})}  (|{g}_{ij}|^{\frac{1}{2}} \beta_{j} w)\circ \varphi^{-1}_{j}(y)dy, $$
   where $d v_{{g}}(y)=|{g}_{ij}|^{\frac{1}{2}} dy $ is the Riemannian volume element on $(\mathcal{M} , {g}),   {g}_{ij} $ are the components of  the metric  ${g}$ in the local chart $(\mathcal{U}_{j}, \varphi_{j})_{j\in J}, $ and  $\displaystyle dy$ is the Lebesgue volume of $\mathbb{R}^{N}.$
 \end{defn}
 \begin{defn}(see \cite{aubin})
 Let $\gamma:[{a},{b}] \rightarrow\mathcal{M}$ be a  differentiable curve in   $\mathcal{M}$ such  that $\gamma \in C^{1}([{a}, {b}], \mathcal{M}).$ Then the length of $\gamma $ is given by  $$ L(\gamma)=\int_{{a}}^{{b}}  ({g}(\gamma^{'}(t), \gamma^{'}(t)))^{\frac{1}{2}} dt.$$
  \end{defn}
  \begin{defn}(see \cite{aubin})
   For any $(y, z)\in  \mathcal{M}^{2},$ we define the distance $d_{{g}}(y, z)$ between  $y$ and $z$ as follows 
  $$d_{{g}}(y, z)= \inf \{   L(\gamma):  \gamma({a})= y, ~   \gamma({b})= z  \}.$$
   \end{defn} 
 \begin{thm} (Stine's theorem \cite{aubin})
For any $({a}, {b})\in  \mathcal{M}^{2},$  $ d_{{g}}({a}, {b}) $  defines a distance on $(\mathcal{M} , {g}), $ and the topology determined by $d_{{g}}({a}, {b}) $ is equivalent to the topology  of $\mathcal{M}$ as a manifold.
 \end{thm}
 
 Next, we recall basic definitions and preliminary facts on the generalized Lebegue spaces $L^{ q({x})}(\mathcal{U})$ on compact manifolds,
where  $\mathcal{U}$ is an open subset of manifold $\mathcal{M}.$ For more background, we refer  to  \cite{aberqi, aubin, hebey}.
  We need to recall the notion of the covariant derivative.
\begin{defn}(see \cite{hebey})
Let  $\nabla$ be the Levi-Civita connection. For $\displaystyle w \in C^\infty(M)$, $\displaystyle \nabla^{k} w$ denotes the $k$-th covariant
derivative of $\displaystyle w$. In local coordinates, the pointwise norm of $\displaystyle \nabla^{k} w$
is given by $$
 \left |\nabla^{k} w \right| = g^{i_1j_1}\cdot\cdot\cdot g^{i_kj_k} (\nabla^{k} w)_{i_1i_2\dots i_k} (\nabla^{k} w)_{j_1j_2\dots j_k}.$$
When $k=1$, the components of $\displaystyle \nabla w$  in local coordinates are given by
$(\nabla w)_i=\nabla^{i} w.$ 
By definition, one has that
 $$
   \left |\nabla w \right| = \sum_{i,j=1}^\infty g^{ij}\nabla^{i}\ w\nabla^{j} w.
$$
 \end{defn}
 We consider the set:
 $$C^{+}(\mathcal{M})=\{ q: \mathcal{M} \to  \mathbb{R}^{+}:    q  \text { is continuous and  } 1< q^{-}< q(y)< q^{+} <+\infty\},$$  $\text  { for every  } y \in   \mathcal{M} ,~ \displaystyle q^{-}= \min _{y \in \mathcal{M} } q(y), ~
 \displaystyle q^{+}= \max _{y \in \mathcal{M} } q(y).$
 \begin{defn}(see \cite{gorka})
Let $\displaystyle q\in  C^{+}(\mathcal{M}) $ and  $k\in \mathbb{N}$. We define the Sobolev space $L_{k}^{q(y)}(\mathcal{M})$ as the completion of $\displaystyle C_{k}^{q(y)}(\mathcal{M})$ with respect to the norm $\displaystyle |w|_{L_{k}^{q(y)}}(\mathcal{M}), $ where $$\displaystyle C_{k}^{q(y)}(\mathcal{M})=\{  w \in  C^{\infty}(\mathcal{M}):    \left |\nabla^{j} w \right| \in L ^{q(y)}(\mathcal{M}), \text{ for every } j=1, 2, \cdots, k  \}, $$
and   $$  |w|_{L^{q(y)}(\mathcal{M})}=\sum _{j=0}^{k}  |  \nabla ^{j} w|_{L^{q(y)}(\mathcal{M})}, $$ where  $\displaystyle |  \nabla ^{j}w|$ is the k-the covariant  derivative of $\displaystyle w.$
 \end{defn}
 \begin{lem}(see \cite{bennouna})
For every $\displaystyle w\in L^{q(y)}(\mathcal{M}),$  the following properties hold:
\begin{itemize}
\item[i)]  If  $\quad $  $\displaystyle |w|_{L^{q(y)}(\mathcal{M})}<1,$  then  $\displaystyle |w|^{q^{-}}_{L^{q(y)}(\mathcal{M})}\leq  \rho_{q(y)}(w) \leq  |w|^{q^{+}}_{L^{q(y)}(\mathcal{M})}.$
\item[ii)]  If $\quad$ $\displaystyle |w|_{L^{q(y)}(\mathcal{M})}>1,$ then  $\displaystyle |w|^{q^{+}}_{L^{q(y)}(\mathcal{M})}\leq  \rho_{q(y)}(w) \leq  |w|^{q^{-}}_{L^{q(y)}(\mathcal{M})}.$
\item[iii)]   $\displaystyle |w|_{L^{q(y)}(\mathcal{M})}<1,   =1,   >1$ \,if only if \, $\displaystyle \rho_{q(y)}(w)<1,   =1,   >1,$
\end{itemize}
where  $\rho_{q(y)}:L^{q(y)}(\mathcal{M}) \to  \mathbb{R}$  is the mapping defined as follows $$\rho_{q(y)}(w)=\int_{\mathcal{M}}|w(y)|^{q(z)}d v_{{g}}(z).$$
 \end{lem}
 \begin{prop}(see \cite{aberqi})
For every $w$ and  $\displaystyle w_{n} \in  L^{q(y)}(\mathcal{M}),$ the following statements are equivalent:
\begin{itemize}
\item[i)] $\displaystyle \lim_{n\to  +\infty}  |w_{n}-w|_{L^{q(y)}(\mathcal{M})}= 0,$
\item[ii )]  $\displaystyle \lim_{n\to  +\infty}  \rho_{q(y)}(w_{n}-w)=0,$
\item[iii )] $\displaystyle w_{n}\rightarrow w$  in measure on  $\mathcal{M}$  and $\displaystyle \lim_{n \to  +\infty} \rho_{q(y)}(w_{n}) - \rho_{q(y)}w)=0.$
\end{itemize}
 \end{prop}
 \begin{lem} (H\"older's  inequality, see \cite{bennouna}) For every $q \in  C^{+}(\mathcal{M}),$ 
 the following inequality  holds:   $$\displaystyle \vert\int_{\mathcal{M}}{v}(y)w(y)d v_{{g}}(y)\vert \leq \left( \frac{1}{q^{-}}+ \frac{1}{{q^{'}}^{-}}\right) |{v}|_{L^{q(y)}(\mathcal{M})}|w|_{L^{{q^{'}}(y)}(\mathcal{M})},$$
 $ \text{ for every } ({v}, w) \in  L^{q(y)}(\mathcal{M})\times  L^{{q^{'}}(y)}(\mathcal{M}),$ where $\frac{1}{q(y)}+ \frac{1}{{q^{'}}(y)}=1.$
 \end{lem}
 \begin{lem} (Simon's  inequality, see  \cite{simon})   \label{simon}
 For every $\displaystyle y, z \in \mathbb{R}^{N}$, the following  holds:
 \begingroup\makeatletter\def\f@size{9}\check@mathfonts 
$$
\left\{\begin{array}{l}
\vert y-z \vert^{n} \leq c_{n}\left(\vert y\vert ^{n-2} y-\vert z\vert^{n-2} z\right) \cdot(y-z), \ n \geq 2 \\
\vert y-z\vert^{n} \leq C_{n}\left[\left(\vert y\vert^{n-2} y-\vert z\vert^{n-2} z\right) \cdot(y-z)\right]^{\frac{n}{2}}\left(\vert y\vert^{n}+\vert z\vert^{n}\right)^{\frac{2-n}{2}},1<n<2
\end{array}\right.
$$
\endgroup
 where $\displaystyle c_{n}=\left(\frac{1}{2}\right)^{-n}$ and $\displaystyle C_{n}=\frac{1}{n-1}$.    
\end{lem}
  \subsection{Topological Degree Theory}
  Let  $E$  be a real separable Banach space and   $E^{\ast}$  its dual. Given a non-empty set $U \subset E,  $ denote by  $\bar{U}$ and by  $ \partial U $ its closure and boundary, respectively.
\begin{defn}(see \cite{brezis})
Let   $f: U \subset E \rightarrow E^{*}$  be an operator. 
\begin{itemize}
\item[1)]  We say that $ f$  is an $\left(S_{+}\right)$-map if  for 
$\{\{z_{n}\}_{n \in \mathbb{N}}, z\} \subset U,$  we have\\
   $$z_{n} \stackrel{weakly}\rightharpoonup z \,\,   \text{weakly  and} \,\,\, \displaystyle \limsup _{n \rightarrow \infty}\left\langle f z_{n}, z_{n}-z\right\rangle \leq 0 \,\,\,\Rightarrow\,\,\,    z_{n} \rightarrow z\cdot $$    
\item[2)] We say that  $f$ is a quasi-monotone  operator if  for every
$\{\{z_{n}\}_{n \in \mathbb{N}}, z\} \subset U,$  we have\\
   $$z_{n} \stackrel{weakly} \rightharpoonup z \Rightarrow \displaystyle \limsup _{n \rightarrow \infty}\left\langle f z_{n}, z_{n}-z\right\rangle \geq 0.$$
\end{itemize}
\end{defn}
\begin{defn}(Condition $\left(S_{+}\right)_{\mathfrak{B}}$, see \cite{kim})
 Suppose that  $U_{1} \subset E $ is such that $U \subset U_{1}, $   $\mathfrak{B}:U_{1}\to   E^{*} $ is a bounded operator,   and  $f:U \subset  E  \to   E $ is an operator. 
 \begin{itemize}
\item[1)]  We say that  $f$  satisfies condition   $\left(S_{+}\right)_{\mathfrak{B}}$ if for every $\{\{z_{n}\}_{n \in \mathbb{N}}, z\} \subset U,$ the following combined properties\\
 \begin{equation}
 \begin{cases}
  \displaystyle  z_{n} \stackrel{weakly}\rightharpoonup  z  \\ 
 \displaystyle   a_{n}= \mathfrak{B}(z_{n}) \stackrel{weakly}\rightharpoonup  a \\
  \displaystyle  \limsup _{n \rightarrow +\infty}  \left\langle f z_{n}, a_{n}-a\right\rangle \geq0
\end{cases}
 \end{equation}
imply $\displaystyle  z_{n} \rightarrow z.$ 
\item[2)](Property $(Q M)_{\mathfrak{B}}$).  We say that  $f$ satisfies condition  $(Q M)_{\mathfrak{B}}$   if for every $\{\{z_{n}\}_{n \in \mathbb{N}}, z\} \subset U,$ we have\\
  $$\displaystyle z_{n} \stackrel{weakly}\rightharpoonup z\,\,\text{and} \,\, a_{n}=\mathfrak{B} z_{n} \stackrel{weakly}\rightharpoonup a \Rightarrow \displaystyle \limsup _{n \rightarrow \infty}\left\langle f z_{n}, a-a_{n}\right\rangle \geq 0 \cdot$$
\end{itemize}
\end{defn}

We consider the following sets
\begingroup\makeatletter\def\f@size{8}\check@mathfonts 
$$ \mathcal{F}^{\star}_{0}(U)=\{g: U \subset E \rightarrow E^{*} \mid  \text{ g is    demi-continuous, of type } (S_{+}), \text{and is bounded} \}.$$
\endgroup
\begingroup\makeatletter\def\f@size{8}\check@mathfonts 
$$\mathcal{F}_{\mathfrak{B}, 1}(U)
=\left\{g: U \subset E \rightarrow E \mid g\right. \text{ is demi-continuous, bounded, of type}\left.\left(S_{+}\right)_{\mathfrak{B}}  \right\}.$$
\endgroup
\begingroup\makeatletter\def\f@size{8}\check@mathfonts 
  $$\mathcal{F}_{\mathfrak{B}}(U)=\left\{f:U \subset E \rightarrow E \mid f \right. \text{is demi-continuous and satisfies  condition} \left.\left(S_{+}\right)_{\mathfrak{B}}\right\}.$$
\endgroup  
Let $U \subset D_{f}$  and  $\mathfrak{B} \in \mathcal{F}^{\star}_{0}(U)$,  where $D_{f}$ denotes  the domain of $f$.
We denote by $\mathcal{N}$  the collection of all bounded open sets in $E$. The following operators will be considered
$$
\begin{aligned}
&\mathcal{F}_{S_{+}}(E)=\left\{f \in \mathcal{F}^{\star}_{0}(\overline{\omega}) : \omega\in \mathcal{N}\right\}, \\
&\mathcal{F}_{B}(E)=\left\{f \in \mathcal{F}_{\mathfrak{B}, 1}(\overline{\omega}) : \omega \in \mathcal{N}, \mathfrak{B} \in \mathcal{F}^{\star}_{0}(\overline{\omega})\right\}, \\
&\mathcal{F}(E)=\left\{f \in \mathcal{F}_{\mathfrak{B}}(\overline{\omega}) : \omega \in \mathcal{N}, \mathfrak{B} \in \mathcal{F}^{\star}_{0}(\overline{\omega})\right\}\cdot
\end{aligned}
$$
 
\begin{lem} (see \cite{kim})\label{le}
Let $\omega$ be a bounded open set in uniformly convex Banach space E,   $\mathfrak{B}: \overline{\omega} \rightarrow E^{*}$  a bounded operator, and  $f$:$\overline{\omega} \rightarrow E$. Then we have
\begin{itemize}
\item[1)]
 If $f$  is locally bounded and satisfies  condition $\left(S_{+}\right)_{\mathfrak{B}}$ and $\mathfrak{B}$ is continuous, then $f$ has the property $(Q M)_{\mathfrak{B}}$.
\item[2)] The operator $ f$ has the property $(Q M)_{\mathfrak{B}}$, 
if for
all
 $\{\{z_{n}\}_{n \in \mathbb{N}}, z\} \subset U$

  $$\quad z_{n} \rightharpoonup z \,\,  \text{and} \,\, \quad a_{n}= \mathfrak{B}z_{n} \rightharpoonup a  \,\, \Rightarrow \liminf \left\langle f z_{n}, a_{n}-a\right\rangle \geq 0.$$
\item[3)] If operators $f_{1}, f_{2}: \overline{\omega} \rightarrow E $  satisfy 
 $(Q M)_{\mathfrak{B}}$ condition, then  $f_{1}+f_{2}$ and $\alpha f_{1}$   also satisfy  $(Q M)_{\mathfrak{B}}$ condition, for every
 positive numbers $\alpha$.
\item[4)] Let  $f_1:\overline{\omega} \rightarrow E $ be an operator of the type $\left(S_{+}\right)_{\mathfrak{B}}$ 
and $f_2: \overline{\omega} \rightarrow E$   an  operator satisfying the property $(Q M)_{\mathfrak{B}}$. Then $f_1+f_2$ satisfies condition $\left(S_{+}\right)_{\mathfrak{B}}$.
\end{itemize}
\end{lem}
\begin{lem}(see \cite{berkovits})
Let $B$ be a bounded open set in $E,$
 $\mathfrak{B} \in \mathcal{F}^{\star}_{0}(\overline{B})$  continuous, and $g: D_{g} \subset E^{*} \rightarrow E$  a  demi-continuous operator such that $\mathfrak{B}(\bar{B}) \subset D_{g}$. Then the following properties hold:
 \begin{itemize}
\item[a)] If ${g}$ is quasi-monotone operator, then\,\, $\displaystyle {I} +{g} \circ \mathfrak{B} \in \mathcal{F}_{\mathfrak{B}}(\bar{B})$, where  ${I}$ denotes the identity operator.
\item[b)] If ${g} $ is  an operator   of type $\left(S_{+}\right)$, then\,\, $ {g}\circ \mathfrak{B} \in \mathcal{F}_{\mathfrak{B}}(\overline{B})$.
\end{itemize}
\end{lem}
\begin{defn}(see \cite{brouwer})
Let $\displaystyle B\subset E$  be a  bounded open set, $\mathfrak{B} \in \mathcal{F}^{\star}_{0}(\overline{B})$ continuous, 
 and  $\displaystyle f, g\,\, \in \mathcal{F}_{\mathfrak{B}}(\overline{E})$. Then the map   $\displaystyle {H}:[0,1] \times \overline{E} \rightarrow E$ given  by
$$
{H}(s, w)=(1-s) fw+  sg w,~ \text{ for every  }~ (s, w) \in[0,1] \times \overline{B},
$$
is called an admissible affine homotopy. 
\end{defn}
\begin{lem} (see \cite{berkovits})
Let $B\subset E$  be a  bounded open set, $\mathfrak{B} \in \mathcal{F}^{\star}_{0}(\overline{B})$ continuous, and ${ f}, { g} \in \mathcal{F}_{\mathfrak{B}}(\overline{E})$. 
Then the homotopy $ {H}(s,\cdot)$ satisfies   condition $\left(S_{+}\right)_{\overline{B}}.$
 \end{lem} 

\begin{thm}(see \cite{berkovits})\label{th}  There exists a unique degree function
$$
d:\left\{(f, F, a): F \in \mathcal{N},  \mathfrak{B}\in \mathcal{F}^{\star}_{0}(\overline{B}), f \in \mathcal{F}_{\mathfrak{B},1}(\overline{E}), a \notin f(\partial E)\right\} \to  \mathbb{Z}
$$
satisfying the following properties:
\begin{itemize}
\item[1)]   
If  $a \in F$, then  $d({I}, F, a)=1$.
\item[2)] If $G:[0,1] \times \overline{B} \rightarrow F$ is a bounded admissible affine homotopy with a common continuous essential inner map and $b:[0,1] \rightarrow F$ is a continuous mapping in $E$, then
$d(G(x, \cdot), F, b(x))$ is constant for every  $x \in[0,1]$ and $b(x) \notin G(t, \partial F)$.
\item[3)]  
If $F_{1}$ and $F_{2}$ are disjoint open subsets of E with $a \notin f\left(\overline{F} \backslash\left(E_{1} \cup E_{2}\right)\right)$, then 
$$
d(f, F, a)=d\left(f, F_{1}, a\right)+d\left(f, F_{2}, a\right).$$
\item[4)] If $d(f, F, a) \neq 0$, then $fu=a$ has a solution in F.
\end{itemize}
\end{thm}
\section{Fractional Sobolev Spaces with  a General Kernel on Compact Riemannian Manifolds}\label{framework}
In this section, we shall introduce fractional Sobolev spaces with a general kernel and prove several qualitative lemmas.
\begin{defn}
Suppose that $(\mathcal{M} , {g})$  is a Riemannian compact manifold of dimension N,  $\mathcal{U}$ is a smooth open subset of $\mathcal{M},$  $K:\mathcal{U}    \times  \mathcal{U} \to(0, +\infty) $ is a symmetric function satisfying   Lévy-integrability and coercivity  conditions,
$  q\in C^{+}(\mathcal{M}), $ and  $p:\mathcal{U}    \times  \mathcal{U} \to (1, +\infty)$ satisfying conditions  (\ref{l3})-(\ref{l30}). We define  fractional Sobolev space 
$W_{K}^{q(y), p(y,z)}(\mathcal{U})$
with  general kernel 
$K(y,z)$ on  compact manifold $\mathcal{M}$ 
 as the set of all measurable functions $w \in L^{q(y)}(\mathcal{U})$ such that $\displaystyle \int_{\mathcal{M}^2}\frac{|w(y)-w(z)|^{p(y,z)}}{\lambda^{p(y,z)}}  K(y, z) d v_{{g}}(y) d v_{{g}}(z)<\infty, \text { for some } \lambda>0$ 
and endow it with the natural norm $$\|w\|_{K}^{q(y), p(y,z)}(\mathcal{U})=[w]_{K, p(y,z)}+|w|_{q(y)},$$
where 
$$[w]_{K , p(y,z)}=  \inf \left\{\lambda>0 : \int_{\mathcal{M}^2} \frac{|w(y)-w(z)|^{p(y,z)}}{\lambda^{p(y,z)}} K(y, z)  d v_{{g}}(y) d v_{{g}}(z)<1\right\},
$$ is the  Gagliardo seminorm of ${u}$ and   $(L^{q(y)}(\mathcal{U}), |.|_{q(y)})$ is a variable exponent Lebesgue space.\end{defn}
\begin{lem}\label{completness}
Suppose that $(\mathcal{M} , {g})$  is a Riemannian compact manifold of dimension N,  $\mathcal{U}$ is a smooth open subset of $\mathcal{M},$  $K:\mathcal{U}    \times  \mathcal{U} \rightarrow (0, +\infty) $ is a  symetric function satisfying  Lévy-integrability and coercivity  conditions,   $q\in C^{+}(\mathcal{M}), $ and  $p:\mathcal{U}    \times  \mathcal{U} \rightarrow (1, +\infty)$ satisfies conditions  (\ref{l3})-(\ref{l30}). Then  $( W_{K}^{q(y), p(y,z)}(\mathcal{U}), ||\cdot||)$  is a Banach space.
\end{lem}
\begin{proof}
Let $\{ w_{n}\}_{n\in \mathbb{N}}$ be a Cauchy sequence  in $W_{K}^{q(y), p(y,z)}(\mathcal{U}).$ For any $\varepsilon>0,$  there exists $N_{\varepsilon}\geq 0$,  such that  for every $n, m\in \mathbb{N},$  $n, m\geq  N_{\varepsilon},$
\begin{equation}\label{k30} |w_{n}-w_{m}|_{q(y)}\leq \|w_{n}-w_{m}\|_{W_{K}^{q(y), p(y,z)}(\mathcal{U})}\leq \varepsilon.  
\end{equation}

Since $(L^{q(y)}(\mathcal{U}), |.|_{q(y)})$ is a Banach space, there exists $w\in  L^{q(y)}(\mathcal{U})$  such that $w_{n}\rightarrow w $  strongly  in $L^{q(y)}(\mathcal{U})$ as $n \rightarrow +\infty. $ Thanks to the converse of the Dominated Convergence Theorem, it follows that for a subsequence still denoted  $\{w_{n}\},$ we have that $w_{n} \rightarrow w $ as  $n \rightarrow +\infty $ a.e on $\mathcal{U}.$ 

Let   $\ell$ be an integer such that $\ell \geq  N_{\varepsilon}$ and $w_{\ell}\in  W_{K}^{q(y), p(y,z)}(\mathcal{U}).$ We use Fatou's lemma  and inequality \eqref{k30} to get
\begin{align*}
&\int_{\mathcal{U} \times \mathcal{U}}|w(y)-w(z)|^{p(y,z)} K(y, z) d v_{{g}}(y) d v_{{g}}(z) \\
&\leq  \liminf _{n \rightarrow +\infty}  \int_{\mathcal{U} \times \mathcal{U}}|w_{n}(y)-w_{n}(z)|^{p(y,z)} K(y, z) d v_{{g}}(y) d v_{{g}}(z)\\
&\leq 2^{p^{+}-1} \liminf _{n \rightarrow +\infty}(|w_{k}(y)-w_{\ell}(y)-(w_{k}(z)-w_{\ell}(z)|^{p(y,z)}K(y, z)\\
&+   |w_{\ell}(y)-w_{\ell}(z)|^{p(y,z)}      K(y, z) )d v_{{g}}(y) d v_{{g}}(z) \\
&\leq 2^{p^{+}} \liminf _{n \rightarrow +\infty} (   \|w_{\ell}-w_{k}\|^{p^{+}}_{W_{K}^{q(y), p(y,z)}(\mathcal{U})}+  \|w_{\ell}-w_{k}\|^{p^{-}}_{W_{K}^{q(y), p(y,z)}(\mathcal{U})}\\
&+\|w_{\ell}\|^{p^{-}}_{W_{K}^{q(y), p(y,z)}(\mathcal{U})}+  \|w_{\ell}\|^{p^{+}}_{W_{K}^{q(y), p(y,z)}(\mathcal{U}))} < +\infty.
\end{align*}

Thus,   $w \in  W_{K}^{q(y), p(y,z)}(\mathcal{U}).$  We combine with Fatou's lemma and  inequality~\eqref{k30}, and obtain  $$ \|w_{n}-w\|^{p^{-}}_{W_{K}^{q(y), p(y,z)}(\mathcal{U})} \leq  \liminf _{k \rightarrow +\infty}  \|w_{n}-w_{k}\|^{p^{-}}_{W_{K}^{q(y), p(y,z)}}(\mathcal{U})\leq \varepsilon.$$ 

 That is,   $w_{n}\rightarrow w $ in $W_{K}^{q(y), p(y,z)}(\mathcal{U})$ as $n \rightarrow +\infty$.
\end{proof}
\begin{lem}\label{unif-convex}
 Suppose that $(\mathcal{M} , {g})$  is a compact Riemannian manifold with $\dim  \mathcal{M}=N,$ 
 $\mathcal{U}$ is  a smooth open subset of $\mathcal{M},$  $K:\mathcal{U}\times \mathcal{U} \rightarrow (0, +\infty) $ is a symmetric function satisfying  Lévy-integrability and coercivity  conditions,   $q\in C^{+}(\mathcal{M}),$ and  $p:\mathcal{U}\times\mathcal{U} \rightarrow (1, +\infty)$ satisfies conditions  \eqref{l3}-\eqref{l30}. Then  $( W_{K}^{q(y), p(y,z)}(\mathcal{U}), ||\cdot||) $ is a  uniformly convex space.  
\end{lem}
\begin{proof}
Let $w, {v}\in W_{K}^{q(y), p(y,z)}(\mathcal{U}),$  and $\eta \in (0, 2)$, such that
$$1= \|{v}\|_{W_{K}^{q(y), p(y,z)}(\mathcal{U})}= \|w\|_{W_{K}^{q(y), p(y,z)}(\mathcal{U})}~
\hbox{and}~
 \|w-{v}\|_{W_{K}^{q(y), p(y,z)}(\mathcal{U})}\geq \eta.$$ 
 
Case 1: $p^{-}\geq 2.$ Thanks to    \cite[Inequality 28]{adams}, we get 
\begin{align}\label{inequality miknosci}
\begin{split}
\displaystyle \|\frac{w-{v}}{2}\|^{p(y,z)}_{W_{K}^{q(y), p(y,z)}(\mathcal{U})}+ \| \frac{w+{v}}{2}\|^{p(y,z)}_{W_{K}^{q(y), p(y,z)}(\mathcal{U})}\\ \leq \frac{1}{2}\left( \| w\|^{p(y,z)}_{W_{K}^{q(y), p(y,z)}(\mathcal{U})}+ \| {v}\|^{p(y,z)}_{W_{K}^{q(y), p(y,z)}(\mathcal{U})}\right).
\end{split}
\end{align}

Thanks to \eqref{inequality miknosci}, it follows that $$\displaystyle \| \frac{w+{v}}{2}\|^{p(y,z)}_{W_{K}^{q(y), p(y,z)}(\mathcal{U})}\leq 1 - (\frac{1}{\eta})^{p^{+}}.$$ 

We take  $\displaystyle  \delta=\delta(\eta)$ such that $$\displaystyle  1-(\frac{\eta}{2})^{p(y,z)}=(1-\delta)^{p(y,z)}$$
and get  $$\displaystyle \| \frac{w+{v}}{2}\|^{p(y,z)}_{W_{K}^{q(y), p(y,z)}(\mathcal{U})}\leq 1-\eta.$$

Case 2: $\displaystyle 1<p(y,z)<2.$ Note that 
\begingroup\makeatletter\def\f@size{8}\check@mathfonts 
$$\|w\|^{p^{'}(y,z)}_{W_{K}^{q(y), p(y,z)}(\mathcal{U})}=   \left[ \int_{\mathcal{M}^2}(|w(y)-w(z)|K(y, z) )^{p ^{'}(y,z)(p(y,z)-1)} d v_{{g}}(y) d v_{{g}}(z)\right]  ^{\frac{1}{p(y,z)-1}},$$
with,
$\displaystyle \frac{1}{p(y,z)}+\frac{1}{p{'}(y,z)}=1.$
\endgroup
 By the reverse Minkowski's inequality \\ \cite[Theorem 2.13]{adams}
  and  \cite[Inequality 27]{adams}, we have  that 
  \begingroup\makeatletter\def\f@size{8}\check@mathfonts 
\begin{align*}
&\displaystyle\|\frac{w +{v}}{2}\|^{p^{'}(y,z)}_{W_{K}^{q(y), p(y,z)}(\mathcal{U})}+\| \frac{w -{v}}{2}\|^{p^{'}(y,z)}_{W_{K}^{q(y), p(y,z)}(\mathcal{U})}\\
&\leq \Big\{  \int_{\mathcal{M}^2} \left[ |w(y)- w(z)+({v}(y)-{v}(z))|^{p^{'}(y,z)}+ |w(y)- w(z)+({v}(y)-{v}(z))|^{p^{'}(y,z)}\right] 
\\
&\times\left( \frac{1}{2}K(y, z)\right)^{p^{'}(y,z)(p(y,z)-1)}d v_{{g}}(y)d v_{{g}}(z) \}^{\frac{1}{p(y,z)-1}}\Big\}\\
&\leq \frac{1}{2}\|w\|^{p(y,z)}_{W_{K}^{q(y), p(y,z)}(\mathcal{U})})+ \frac{1}{2}\|{v}\|^{p(y,z)}_{W_{K}^{q(y), p(y,z)}(\mathcal{U})})^{p^{'}(y,z)-1}=1,
\end{align*}
\endgroup
therefore $$\displaystyle \| \frac{w+{v}}{2}\|^{p^{'}(y,z)}_{W_{K}^{q(y), p(y,z)}(\mathcal{U})}\leq 1- (\frac{1}{\varepsilon})^{p^{'}(y,z)}.$$ 

To complete the argument, choose $\displaystyle \delta=\delta (\eta) $ such that $\displaystyle 1-(\frac{\eta}{2})^{p^{'}(y,z)}= (1-\delta)^{p^{'}(y,z)}\cdot$
\end{proof}
\begin{rem}
According to  the Milman-Petits theorem \cite{adams},  $\displaystyle W_{K}^{q(y), p(y,z)}(\mathcal{U})$ is  a reflexive space.
\end{rem}
\begin{lem}\label{separability}
 Suppose that $(\mathcal{M}, {g})$  is a compact Riemannian manifold with $\dim\mathcal{M}=N$,   $ \mathcal{U}$   is a smooth open subset of $\mathcal{M},$  $K:\mathcal{U}    \times  \mathcal{U} \rightarrow (0, +\infty) $  is a  symetric function satisfying conditions Lévy-integrability and coercivity conditions,   $q\in C^{+}(\mathcal{M}), $ and  $p:\mathcal{U}    \times  \mathcal{U} \rightarrow (1, +\infty)$ satisfies conditions  \eqref{l3}-\eqref{l30}. Then  $( W_{K}^{q(y), p(y,z)}(\mathcal{U}), ||\cdot||) $ is a  separable space.  
\end{lem}
\begin{proof}
Let  $\displaystyle L:W_{K}^{q(y), p(y,z)}(\mathcal{U})\rightarrow L ^{q(y)}(\mathcal{M})\times L ^{p(y,z)}(\mathcal{M}\times\mathcal{M})$ be  the operator defined by $$L(w)= \big(w(y),  (w(y)-w(z))K^{\frac{1}{p(y,z)}}(y,z)\big).$$  
$\displaystyle L$ is clearly well-defined and is an isometry. By   \cite[Proposition 3.17]{brezis}, the space  $\displaystyle ( W_{K}^{q(y), p(y,z)}(\mathcal{U}), ||\cdot||) $ is  indeed  separable.
\end{proof}
\section{Proof of Theorem ~\ref{embedding}}\label{1}
In this section we shall prove Theorem~\ref{embedding}, establishing an embedding of $W_{K}^{q(y), p(y,z)}(\mathcal{U})$ into $L ^{\mathfrak{\ell}(y)}(\mathcal{U})$.

\begin{proof}
Let  $\displaystyle p, q,$ and $\displaystyle \mathfrak{\ell}$  be continuous functions and $\mathcal{U}$  an open subset of $\mathcal{M}.$ There exist two positive constants $\displaystyle \alpha_{1}$ and  $\displaystyle \alpha_{2}$ such that 
\begin{equation} \label{ab}
\displaystyle q(y)\geq  p(y, y)+\alpha_{1}>0
\end{equation}
and 
\begin{equation} \label{ac}
 \frac{N p(y, y)}{N- sp(y, y)}\geq  \mathfrak{\ell}(y)+\alpha_{2}>0,
\end{equation}
for every $y\in \overline{\mathcal{U}}.$
Let $t\in(0, s). $ We use the continuity of $\displaystyle  p,  q, \mathfrak{\ell},$   \eqref{ab}, and \eqref{ac} to find a constant $\displaystyle \varepsilon=\varepsilon( p,  q,  \mathfrak{\ell}, \alpha_{1}, \alpha_{2})$ and a finite family of disjoint Lipschitz sets $\mathcal{U}_{j}$ such that
$\mathcal{U}= \bigcup_{j=1}^{N}  \mathcal{U}_{j}  \text{  and } \hbox{diam} ( \mathcal{U}_{j})< \varepsilon ,$~
 $$ \frac{N p({m}, z)}{N- tp({m}, z)}\geq  \mathfrak{\ell}(y)+ \frac{\alpha_{2}}{2}>0,~
 q(y)\geq  p({m}, z)+ \frac{\alpha_{1}}{2}>0,~\forall
 (y, z, {m}) \in  \mathcal{U}_{j} ^{3}.$$
 
 Put  $\displaystyle p_{j}= \inf_{(y, z)\in  \mathcal{U}_{j} \times \mathcal{U}_{j}} \{p(y, z)-\delta \}.$ By the continuity of the involved exponents, we can choose   $\displaystyle \delta=\delta(\alpha_{2}), $ with $\displaystyle p^{-}-1>\delta >0$ such that $$\frac{N p_{j}}{N-tp_{j}} \geq  \mathfrak{\ell}(y)+ \frac{\alpha_{2}}{3}, ~ \hbox{for every} ~ y\in  \mathcal{U}_{j}.$$
 
 So, we have the following
 \begin{equation}\label{statement1}
   \displaystyle \frac{N p_{j}}{N-tp_{j}} \geq  \mathfrak{\ell}(y)+ \frac{\alpha_{2}}{3}, 
 \,\,\text{for every}\,\, \displaystyle y\in  \mathcal{U}_{j}.\\
 \end{equation}
 \begin{equation}\label{statement2}
  \displaystyle p_{j}+\frac{\alpha_{1}}{2} \leq  q(y),  \,\,\text{for every}\,\, y\in \mathcal{U}_{j}.
 \end{equation} 
 
 By \cite[Lemma 2.4 ]{guo}, there exists a constant $\displaystyle C=C(N, t, \varepsilon, p_{j},   \mathcal{U}_{j})$ such that (see  \cite{guo} for more details)
 \begin{equation} \label{guo}
\displaystyle  ||w||_{L^{p_{j}^{*}}( \mathcal{U}_{j})} \leq C(
 ||w||_{L^{{p_{j}}}( \mathcal{U}_{j})} + [w]^{t, p_{j}}(\mathcal{U}_{j})) ~ \text {  for every  } ~ w        \in  W^{s, p_{j}} (\mathcal{U}_{j}).  
 \end{equation}
 
 Now, we shall prove the following three inequalities.\\
a) There exists a constant $\displaystyle c_{1}>0,$  such as: $\displaystyle c_{1} |w| _{L^{\mathfrak{\ell}(y)}(\mathcal{ \mathcal{U}})} \leq \sum_{j=0}^{N}  ||w||_ {L^{p_{j}^{*}}( \mathcal{U}_{j})}.$\\
b) There exists a constant $\displaystyle c_{2}>0,$  such as: $\displaystyle \sum_{j=0}^{N} ||w||_ {L^{p_{j}}( \mathcal{U}_{j})} \leq c_{2}  |w| _{L^{q(y)}(\mathcal{ \mathcal{U}})}.$\\
c) There exists a constant $\displaystyle c_{3}>0,$  such as:  $\displaystyle \sum_{j=0}^{N} [w]^{t, p_{j}}(\mathcal{U}_{j})) \leq  c_{3}[w]^{s, p(y, z)}(\mathcal{U}).$\\

 We shall first prove $(a)$.  We have that  $$\displaystyle w(y)= \sum_{j=0}^{N}  |w(y)|\chi _{\mathcal{U}_{j}},$$
 where  $\displaystyle \chi _{\mathcal{U}_{j}}$ is a characteristic function. Hence, we have  $$|w| _{L^{\mathfrak{\ell}(y)}(\mathcal{ \mathcal{U}})}\leq  \sum_{j=0}^{N}  |w| _{L^{\mathfrak{\ell}(y)}( \mathcal{U}_{j})}.$$ 
  Combining the statement~ \eqref{statement1} with the H\"older inequality, we obtain 
$$
\displaystyle  |w| _{L^{\mathfrak{\ell}(y)}( \mathcal{U}_{j})}\leq  C  ||w|| _{L^{p_{j}^{*}}( \mathcal{U}_{j})} |1|_{L^{{a}_{j}(y)}(\mathcal{U}_{j})} \leq  C(\mathcal{U}_{j}, {a}_{j})  ||w|| _{L^{p_{j}^{*}}(\mathcal{U}_{j})},$$
\begin{equation}
\hbox{where}~ \displaystyle \frac{1}{\mathfrak{\ell}(y)}+ \frac{1}{p_{j}^{*}}= \frac{1}{{a}_{j}(y)}, ~ \hbox{ for every}~
y\in \mathcal{U}_{j}.
  \end{equation}   
  Similarly, by using the fact that $q(y) > p_j$ for every $\displaystyle y\in \mathcal{U}_{j}, $ we get $(b)$.\\
  
  Now, we show $(c).$ Put $$\displaystyle G(y, z)= \frac{|w(y)- w(z)|}{d_{{g}}(y, z)^{s}}.$$  
   We use the H\"older  inequality and the definition of  $\displaystyle p_{j}, $ to get 
  \begin{align*}
\displaystyle   [w]^{t, p_{j}}(\mathcal{U}_{j})&= \left( \int_{\mathcal{U}_{j} \times \mathcal{U}_{j}}   
  \frac{|w(y)- w(z)| ^{p_{j}}}{d_{{g}}(y, z)^{N+tp_{j}}} d v_{{g}}(y) d v_{{g}}(z)\right) ^{\frac{1}{p_{j}}}\\
  &= \left( \int_{\mathcal{U}_{j}\times \mathcal{U}_{j}}   
  (\frac{|w(y)- w(z)|}{d_{{g}}(y, z)^{s}})^{ p_{j}} \frac{d v_{{g}}(y) d v_{{g}}(z)}{d_{{g}}(y, z)^{N+(t-s)p_{j}}}\right) ^{\frac{1}{p_{j}}}\\
  &\leq  C |G|_{L^{p(y, z)}( \mu _{{g}}, \mathcal{ \mathcal{U}}_{j} \times \mathcal{ \mathcal{U}}_{j}} |1|_{L^{\mathfrak{B}_{j}(y, z)}( \mu _{{g}},  \mathcal{U}_{j} \times  \mathcal{U}_{j})}\\
 &=C (\mathcal{U}_{j}, \mathfrak{B}_{j})|G|_{L^{p(y, z)}( \mu _{{g}}, \mathcal{U}_{j} \times \mathcal{U}_{j})},
  \end{align*}
  where $$\displaystyle \frac{1}{p_{j}}=\frac{1}{p(y, z)}+\frac{1}{ \mathfrak{B}_{j}(y, z)}
  ~\hbox{ and}~
  \displaystyle d\mu _{{g}}(y, z)= \frac{d v_{{g}}(y) d v_{{g}}(z)}{d_{{g}}(y, z)^{N+(t-s)p_{j}}}.$$
  
  Next, we prove that $|G|_{L^{p(y, z)}}\leq C  [w]^{s, p(y, z)}(\mathcal{U}_{j}).$
  Let $\lambda  >0$ be such that $$\int_{\mathcal{U}_{j} \times \mathcal{U}_{j}} \frac{|w(y)- w(z)|^{p(y, z)}}{\lambda ^{p(y, z)}d_{{g}}(y, z)^{N+sp(y, z)} }  d v_{{g}}(y) d v_{{g}}(z)<1.$$  
  Put $\displaystyle k=\sup \{1, \sup_{(y, z)\in \mathcal{U} \times \mathcal{U}} d_{{g}}(y, z)^{s-t} \} $ and $\displaystyle \breve{\lambda}= \lambda k.$ Then we have
  \begin{equation}
 \displaystyle  \int_{\mathcal{U}_{j}^2} \frac{|w(y)- w(z)|^{p(y, z)}}{((d_{{g}}(y, z)^{s} \breve{\lambda})^{p(y, z)}} \frac{d v_{{g}}(y) d v_{{g}}(z)}{d_{{g}}(y, z)^{N+(t-s)p_{j}}}\leq  \int_{\mathcal{U}_{j}^2}  \frac{|w(y)- w(z)|^{p(y, z)}}{d_{{g}}(y, z)^{N+sp(y, z)} \lambda ^{p(y, z)}} <1,
  \end{equation}  
  therefore $$|G|_{L^{p(y, z)}( \mu _{{g}}( \mathcal{U}_{j} \times \mathcal{ \mathcal{U}}_{j})}\leq \lambda k.$$  
   It now follows from \eqref{equation51}, inequalities $(a),   (b), (c), $  and   \cite[Lemma 2.4]{guo} that 
  \begin{align*}
 \displaystyle  |w|_{L^{\mathfrak{\ell}(y)}(\mathcal{U})}&\leq c \sum_{j=0}^{N} ||{u}||_{L^{p_{j}^{*}}(\mathcal{U})}\\
 &\leq c \sum_{j=0}^{N} \left( ||w||_{L^{ p_{j}}(\mathcal{U}_{j})}+ [w]^{t, p_{j}}(\mathcal{U}_{j})\right) \\
  &\leq c \left( |w|_{L^{q(y)}(\mathcal{U})}+ [w]^{s, p(y,z)}(\mathcal{U})\right) \\
  &\leq  c \alpha_{0}  \left( |w|_{L^{q(y)}(\mathcal{U})}+[w]^{  {K}, p(y,z)}(\mathcal{U})\right) \\
  &=  c \alpha _{0}  ||w||_{W_{K}^{q(y), p(y,z)}(\mathcal{U})}\cdot
  \end{align*}
We show that this embedding is compact. Let $\{w_{n} \}$ be a  bounded sequence in $W_{K}^{q(y), p(y,z)}(\mathcal{U}),$ we need to prove that there exists $w \in  L^{\mathfrak{\ell}(y)}(\mathcal{M})$ such that   for every $\mathfrak{\ell}(y)\in(1, p^{*}_{s}),$
 $$w_{n}  \to   w \text{ in  }   L^{\mathfrak{\ell}(y)}(\mathcal{M})\,\,\, \text { as }\,\,\, n \to  +\infty. $$  
  Since $\mathcal{M}$ is a compact Riemannian N-manifold, we can cover $\mathcal{M}$ by a finite number of charts $\displaystyle (\mathcal{U}_{j},  \varphi_{j})_{j=1, \ldots, m}$ satisfying 
  $$\frac{1}{Q} \delta_{ij}\leq  g_{ij}^{s}  \leq Q  \delta_{ij},$$
where  $g_{ij}^{s}$  are bilinear forms and $Q>1.$  Let  $\eta _{j}$ be a smooth partition of unity subordinate to the chart $\displaystyle (\mathcal{U}_{j},  \varphi_{j})_{j=1, \ldots, m}.$ Let  $\displaystyle w_{n} \in  
W_{K}^{q(y), p(y,z)}(\mathcal{M}).$ 
Then $$\displaystyle \eta _{j}w_{n} \in  
W_{K}^{q(y), p(y,z)}(\mathcal{M})~\hbox{ and}~\displaystyle (\eta _{j}w_{n}) \circ \varphi _{j}^{-1} \in W_{K}^{q(y), p(y,z)}(B_{0}(1)),$$ where $B_{0}(1)$ is an open unit ball of $\mathbb{R}^{N}.$ By \cite[Theorem 1.1]{kaufmann}, there  exists $w_{j} \in L^{\mathfrak{\ell}(y)}(\varphi _{j}(\mathcal{U}_{j})) $ such that  $$(\eta _{j}w_{n})\circ \varphi _{j}^{-1} \to   w_{j}\  \text{ strongly  in } \  L^{\mathfrak{\ell}(y)}(\varphi _{j}(\mathcal{U}_{j})) \ \text {  as }\  n \to  +\infty.$$
Hence
  $$\displaystyle \eta _{j}w_{n} \to   w_{j}  \circ \varphi _{j}= {a}_{j} \, \hbox{strongly  in} \,
  \displaystyle L^{\mathfrak{\ell}(y)}(\varphi _{j}(\mathcal{U}_{j})) ~ \hbox{as}~ 
  n \to  +\infty.$$  
   Finally, we put $$\displaystyle w= \sum_{j=1}^{m} {a}_{j} =  \sum_{j=1}^{m} w_{j} \circ \varphi _{j}  \in L^{\mathfrak{\ell}(y)}(\mathcal{M}).$$
\end{proof}
\vfill \eject
\section{Proof of Theorem~\ref{existence}}\label{2}
In this section, we shall prove our existence result stated in Theorem~\ref{existence}.

\begin{defn}
 We say that $w\in W_{K}^{q(y), p(y,z)}(\mathcal{U}) $  is a weak solution of  problem~\eqref{k1} if for every $ h\in (W_{K}^{q(y), p(y,z)}(\mathcal{U}))^{*}$ we have,
 \begin{align*}
&\int_{\mathcal{M}\times \mathcal{M}}|w(y)-w(z)|^{p(y, z)-2}(w(y)-w(z))(h(y)-h(z)) K(y, z) d v_{{g}}(y) d v_{{g}}(z)\\
&=
\lambda \int_{\mathcal{M}} \beta(y)|w (y)|^{r(y)-2}  w(y)h(y) d v_{{g}}(y)+\int_{\mathcal{M}} f(y, w(y)) h(y) d v_{{g}}(y).
 \end{align*} 
\end{defn}
 We consider the functional 
 $\zeta:W_{K}^{q(y), p(y,z)}(\mathcal{U})\rightarrow \mathbb{R} $  defined by 
\begin{align*}
 \zeta(w)&=\int_{\mathcal{M} \times \mathcal{M}}\frac{1}{p(y, z)}|w(y)-w(z)|^{p(y, z)} K(y, z) d v_{{g}}(y) d v_{{g}}(z)
\\
  &-\lambda \int_{\mathcal{M}}\frac{1}{r(y)}  \beta(y)|w(y)|^{r(y)}  d v_{{g}}(y)-\int_{\mathcal{M}} F(y, w(y)) d v_{{g}}(y).
 \end{align*}
 Then it follows  from  \cite{bahrouni, kaufmann} that $\zeta \in C^{1}(W_{K}^{q(y), p(y,z)}(\mathcal{U}), \mathbb{R})$ and   
  \begin{align*}
 &\langle \zeta ^{'}(w) , h\rangle\\&=\int_{\mathcal{M}^2}|w(y)-w(z)|^{p(y, z)-2}(w(y)-w(z))(h(y)-h(z)) K(y, z) d v_{{g}}(y) d v_{{g}}(z)\\
&-\lambda \int_{\mathcal{M}} \beta(y)|w(y)|^{r(y)-2}  w(y) h(y) d v_{{g}}(y)-\int_{\mathcal{M}} f(y, w(y)) h(y) d v_{{g}}(y)\\
 &=< L(w), h> - <S _{1}(w), h >- < S _{2}(w), h>.
 \end{align*}
\begin{proof}
Let $\displaystyle w\in W_{K}^{q(y), p(y,z)}(\mathcal{U}).$ Then $w$ is a weak solution of  problem~\eqref{k1}  if and only if 
\begin{equation}\label{abtractequation}
\displaystyle Lw+Sw=0,
\end{equation} 
where $\displaystyle L, \,S$ are  the operators defined in Lemma \ref{a} and  Lemma \ref{l2}. Since $\displaystyle S$ is bounded, continuous, and quasi-monotone (see Lemma \ref{a}) and $\displaystyle L$ is  strictly monotone,  thanks to the Minty-Browder  Theorem \cite[ Theorem 26 A]{zeidler},   we have that $\displaystyle L^{-1}=G$ is bounded continuous of type  $\displaystyle (S_{+}).$ 

Equation \eqref{abtractequation}  is equivalent to
\begin{equation}\label{e2} 
\displaystyle w= Gh \,\,\, \text{ and }\,\,\, h+S\circ Gh=0.
\end{equation}
To solve \eqref{e2}, we shall use the Berkovits topological degree introduced in Section~\ref{rappel}. To this end, we first show that the  set 
$$D=\left\lbrace h\in (W_{K}^{q(y), p(y,z)}(\mathcal{U}))^{*}: h+ tS\circ Gh=0 \, \text { for some  }\, t\in [0, 1]   \right\rbrace ,$$ is bounded. 
Let $\displaystyle h\in D $  and take $\displaystyle w= Gh. $ Using the  growth condition $(\mathcal{B}_{1})$, the H\"older inequality, the Young inequality, and continuous embedding \\
$\displaystyle  W_{K}^{q(y), p(y,z)}(\mathcal{U}) \hookrightarrow   L^{q(y)}(\mathcal{U}),$ we get 
\begin{align*}
&\displaystyle ||G h||_{W_{K}^{q(y), p(y,z)}(\mathcal{U})}\leq  \int_{\mathcal{M} \times \mathcal{M}}\left|w(y)-w(z)\right|^{p(y, z)} K(y,z) d v_{{g}}(y) d v_{{g}}(z)\\
&\displaystyle = \left\langle Lw, h\right\rangle \displaystyle = \left\langle h, Gh\right\rangle \\
&\displaystyle \leq |t| \left\langle  S\circ Gh, Gh\right\rangle \\
&\leq \lambda  \int_{\mathcal{M} } \beta(y) |w(y)|^{r(y)} d v_{{g}}(y)+ \int_{\mathcal{M}} f(y, w(y))w(y) d v_{{g}}(y)\\
&\displaystyle \leq  \lambda  ||\beta||_{\infty} C_{1} ||w||^{r^{+}}_{W_{K}^{q(y), p(y,z)}(\mathcal{U})}+C_{2}\left( \int_{\mathcal{M}} |f(y, w)|^{q^{'}(y)}d v_{{g}}(y)\right) ^{\frac{1}{q^{'}(y)}}\\
&+ C_{3}\left( \int_{\mathcal{M}}  |w(y)|^{q(y)} d v_{{g}}(y)\right) ^{\frac{1}{q(y)}}\\
& \displaystyle \leq \lambda  ||\beta||_{\infty} C_{1} ||w||^{r^{+}}_{W_{K}^{q(y), p(y,z)}(\mathcal{U})}+  C_{2} ||\beta||_{\infty} \int_{\mathcal{M}} \left( (1+|w(y)|^{(q(y)-1)q^{'}(y)})\right) ^{\frac{1}{q^{'}(y)}}\\
&\displaystyle + C_{3} \left( \int_{\mathcal{M}} |w(y)|^{q(y)} d v_{{g}}(y) \right) ^{\frac{1}{q(y)}}\\
&\displaystyle  \leq \lambda  ||\beta||_{\infty} C_{1} ||w||^{r^{+}}_{W_{K}^{q(y), p(y,z)}(\mathcal{U})}+2^{q^{+}}C_{4}||\beta||_{\infty}  ||w||_{W_{K}^{q(y), p(y,z)}(\mathcal{U})}\\
&\displaystyle +C_{4}||w||_{W_{K}^{q(y), p(y,z)}(\mathcal{U})}. 
\end{align*}

Since $\displaystyle S$ is bounded, it follows that $\displaystyle D$ is bounded in $\displaystyle (W_{K}^{q(y), p(y,z)}(\mathcal{U}))^*$.
As a result, there exists a positive constant $\displaystyle \eta>0$ such that $$ ||h||_{W_{K}^{q(y), p(y,z)}(\mathcal{U})^{*} }<\eta,\,\,\, \text { for every  } \,\, h\in D.$$

Furthermore, $\displaystyle h+ tS\circ Gh \neq 0$ for every $\displaystyle (h, t)\in \partial B_{\eta}(0)  \times  [0, 1]. $ Using Lemma \ref{le}, and $\displaystyle i+ S\circ G\in \mathcal{F}_{\mathfrak{B}}(\overline{B_{\eta}(0)})$  and $\displaystyle i=L\circ G \in \mathcal{F}_{\mathfrak{B}}(\overline{B_{\eta}(0)})$  are present. Next, $\displaystyle i+ S\circ G  $ is also bounded because the operators $\displaystyle i,\, S$ and $\displaystyle G$  are all bounded. We come to the conclusion that 
 $$i+ S\circ G \in  \mathcal{F}_{\mathfrak{B}, 1}(\overline{B_{\eta}(0)}) \text { and } i\in  \mathcal{F}_{\mathfrak{B}, 1}(\overline{B_{\eta}(0)}).$$
 
We consider the  map  $\displaystyle {H}:[0, 1]\times  \overline{B_{\eta}(0)} \to   (W_{K}^{q(y), p(y,z)}(\mathcal{U}))^{*}$ given by $$H(t, w)=  w +t S\circ L w.$$ 

 By the statements (1)-(2) in Theorem~\ref{th}, we can deduce   $$d(i+ S\circ G,  B_{\eta}(0), 0)= d(i, B_{\eta}(0), 0)=1,$$ by applying the homotopy invariance and normalization properties of the degree $d$ from Theorem \ref{th}. Therefore there exists $\displaystyle w \in  B_{\eta}(0)$ such that  $h+S\circ G h=0.$
We can now deduce that  $\displaystyle w= Gh$  is a weak solution to problem \eqref{k1}  in $W_{K}^{q(y), p(y,z)}(\mathcal{U}).$ This completes the proof.
\end{proof}
\section{Appendix}\label{appendix}
 \begin{lem}\label{a}
Suppose that $(\mathcal{M} , {g})$  is  a compact $N$-dimensional Riemannian manifold,   
 $\mathcal{U}$  is  a smooth open subset of $\mathcal{M},$ 
 $K:\mathcal{U}\times\mathcal{U} \rightarrow (0, +\infty) $ is a  symmetric function satisfying  Lévy-integrability and coercivity  conditions. Assume that assumption $ (\mathcal{B}_{1})$  holds. Then the operator  $L:W_{K}^{q(y), p(y,z)}(\mathcal{U})\rightarrow (W_{K}^{q(y), p(y,z)}(\mathcal{U}))^{*} $  is continuous, bounded and strictly monotone, and
 \begin{itemize}
\item[i)] L is  an operator of type $\displaystyle (S_{+}),$ 
  \item[ii)]   $\displaystyle L:W_{K}^{q(y), p(y,z)}(\mathcal{U})\rightarrow (W_{K}^{q(y), p(y,z)}(\mathcal{U}))^{*} $ is a homeomorphism.
\end{itemize}
 \end{lem}
 \begin{proof}
 It is obvious that $\displaystyle  L$ is bounded.
 We show that $L$ is continuous. Assume that $w_{n}  \rightarrow w$ in $W_{K}^{q(y), p(y,z)}(\mathcal{U})$ and we show that $\displaystyle L(w_{n}) \rightarrow L(w)$ in $\displaystyle (W_{K}^{q(y), p(y,z)}(\mathcal{U}))^{*}$. Indeed,
\begingroup\makeatletter\def\f@size{10.1}\check@mathfonts
\begin{align*}
 &\displaystyle \left\langle Lw_{n}-Lw, \varphi\right\rangle =
\int_{ \mathcal{M}^2 } \Big[(\left|w_{n}(y)-w_{n}(z)\right|^{p(y, z)-2}\left(w_{n}(y)-w_{n}(z)\right)\\
&-|w(y)-w(z)|^{p(y, z)-2}(w(y)-w(z))\Big]\times K(y, z)(\varphi(y)-\varphi(z))d v_{{g}}(y) d v_{{g}}(z) \\
&=\int_{\mathcal{M}^2}\Big[\left|w_{n}(y)-w_{n}(z)\right|^{p(y, z)-2}\left(w_{n}(y)-w_{n}(z)\right)K(y,z)^{\frac{p(y,z)-1}{p(y, z)}}\\
&-|w(y)-w(z)|^{p(y, z)-2}(w(y)-w(z))K(y,z)^{\frac{p(y,z)-1}{p(y, z)}}\Big] \\
&\times (\varphi(y)-\varphi(z))K(y,z)^{\frac{1}{p(y, z)}} d v_{{g}}(y) d v_{{g}}(z). 
\end{align*}
\endgroup
Put 
\begingroup\makeatletter\def\f@size{9}\check@mathfonts
 $$G_{n}(y,z)=|w_{n}(y)-w_{n}(z)|^{p(y, z)-2}(w_{n}(y)-w_{n}(z))K(y,z)^{\frac{p(y,z)-1}{p(y, z)}}\in L ^{p^{'}(y,z)}( \mathcal{U}\times \mathcal{U} ),$$
\endgroup 
$$G(y,z)=|w(y)-w(z)|^{p(y, z)-2}(w(y)-w(z))K(y,z)^{\frac{p(y,z)-1}{p(y, z)}}\in L ^{p^{'}(y,z)}(\mathcal{U} \times \mathcal{U} ),$$ 
and
$$F(y,z)= (\varphi(y)-\varphi(z)) K(y,z)^{\frac{1}{p(y, z)}}\in L ^{p(y,z)}(\mathcal{U} \times \mathcal{U} ),$$
where, $\displaystyle \frac{1}{p(y,z)}+ \frac{1}{p^{'}(y,z)}=1.$
Thanks to the H\"older   inequality, we have $$\left\langle L\left(w_{n}\right)-L(w), \varphi\right\rangle \leq 2 |G_{n}(y,z)-G(y,z)|_{L ^{p^{'}(y,z)}(\mathcal{U} \times \mathcal{U} )} |F|_{L ^{p(y,z)}(\mathcal{U}\times \mathcal{U}}).$$

Thus, $$||L\left(w_{n}\right)-L(w)||_{(W_{K}^{q(y), p(y,z)}(\mathcal{U}))^{*}} \leq 2 |G_{n}(y,z)-G(y,z)|_{L ^{p^{'}(y,z)}(\mathcal{U}\times \mathcal{U} )}.$$

Let $$\displaystyle V_{n}(y,z)=(w_{n}(y)-w_{n}(z))K(y,z)^{\frac{1}{p(y, z)}}\in L ^{p(y,z)}(\mathcal{U}\times \mathcal{U}),$$ and  
$$\displaystyle V(y,z)=(w(y)-w(z))K(y,z)^{\frac{1}{p(y, z)}}\in L ^{p(y,z)}(\mathcal{U} \times \mathcal{U}).$$

Since $\displaystyle w_{n}\rightarrow w $ in $\displaystyle W_{K}^{q(y), p(y,z)}(\mathcal{U}),$ we have $\displaystyle V_{n}\rightarrow V$ in  $\displaystyle  L ^{p(y,z)}(\mathcal{U} \times \mathcal{U}).$
So,  there exists a subsequence of $\displaystyle \{V_{n}\}_{n\in \mathbb{N}}$ and  $\displaystyle h(y,z)\in L ^{p(y,z)}(\mathcal{U}\times \mathcal{U})$ such that $\displaystyle V_{n} \rightarrow V$ a.e in $\mathcal{U}\times \mathcal{U}$ and $\displaystyle |V_{n}|\leq h(y,z).$   
Therefore we have $\displaystyle G_{n} \rightarrow G$ a.e in $\mathcal{U}\times \mathcal{U}$ and $$\displaystyle |G_{n}(y,z)|= |V_{n}(y,z)|^{p(y,z)-1}\leq h(y,z)^{p(y,z)-1}.$$

We use the Dominated Convergence Theorem to get  $$G_{n} \rightarrow G \text { in }  L ^{p^{'}(y,z)}(\mathcal{U}\times \mathcal{U} ).$$  

By Lemma~\ref{simon},  $\displaystyle L $ is strictly monotone.
Now, we show that $\displaystyle  L$ is mapping of type $\displaystyle (S_{+}).$  Let   $\displaystyle \{w_{n}\}_{n\in \mathbb{N}} \subset W_{K}^{q(y), p(y,z)}(\mathcal{U}) $ be a sequence with $\displaystyle w_{n} \rightharpoonup w$ in $W_{K}^{q(y), p(y,z)}(\mathcal{U})$ and  
$\displaystyle \limsup _{n \rightarrow+\infty}\left\langle L\left(w_{n}\right)-\right.$ $\left.L(w ), w_{n}-w\right\rangle \leq 0$.
Using (i), we get   
\begin{equation}\label{k3}
\displaystyle 0=\lim _{n \rightarrow+\infty}\left\langle Lw_{n}-Lw, w_{n}-w\right\rangle\cdot
\end{equation}

By Theorem\ref{embedding}, we have that $\displaystyle w_{n}(y) \rightarrow w(y) \,\text { a.e.  in  }  \, \mathcal{U}.$
This, in combination with Fatou's lemma, gives us

$$\liminf _{n \rightarrow +\infty} \int_{\mathcal{M}^2}\left|w_{n}(y)-w_{n}(z)\right|^{p(y, z)} K(y, z) d v_{{g}}(y) d v_{{g}}(z)$$
\begin{equation}
 \geq \int_{\mathcal{M}^2}|w(y)-w(z)|^{p(y, z)} K(y,z) d v_{{g}}(y) d v_{{g}}(z).
\end{equation}

On the other hand, we have
\begin{equation}\label{k4}
\lim _{n \rightarrow+\infty}\left\langle L\left(w_{n}\right), w_{n}-w\right\rangle=\lim _{n \rightarrow+\infty}\left\langle L\left(w_{n}\right)-L(w), w_{n}-w\right\rangle=0.
\end{equation}
Using Young's inequality, we can see  there is a positive constant $\displaystyle c>0$ such that
\begingroup\makeatletter\def\f@size{9}\check@mathfonts
\begin{align*}
&\left\langle L\left(w_{n}\right), w_{n}-w\right\rangle \\&=\int_{\mathcal{M} \times \mathcal{M}}\left|w_{n}(y)-w_{n}(z)\right|^{p(y, z)} K(y,z) d v_{{g}}(y) d v_{{g}}(z) \\
&-\int_{\mathcal{M}^2}\left|w_{n}(y)-w_{n}(z)\right|^{p(y, z)-2}\left(w_{n}(y)-w_{n}(z)\right)(w(y)-w(z)) K(y,z)d v_{{g}}(y) d v_{{g}}(z) \\
& \geq \int_{\mathcal{M}^2}\left|w_{n}(y)-w_{n}(z)\right|^{p(y, z)} K(y,z)d v_{{g}}(y) d v_{{g}}(z) \\
&-\int_{\mathcal{M}^2}\left|w_{n}(y)-w_{n}(z)\right|^{p(y, z)-1}|w(y)-w(z)| K(y,z) d v_{{g}}(y) d v_{{g}}(z)\\
&\geq c \int_{\mathcal{M}^2}\left|w_{n}(y)-w_{n}(z)\right|^{p(y, z)} K(y,z) d v_{{g}}(y) d v_{{g}}(z) \\
&-c \int_{\mathcal{M}^2}|w(y)-w(z)|^{p(y, z)} K(y,z) d v_{{g}}(y) d v_{{g}}(z).
\end{align*}
\endgroup
According to \eqref{k3} - \eqref{k4}, we get  
 
 $$\lim _{n \rightarrow+\infty} \int_{\mathcal{M}^2}\left|w_{n}(y)-w_{n}(z)\right|^{p(y, z)} K(y,z) d v_{{g}}(y) d v_{{g}}(z)$$
 \begin{equation} \label{k5}
 =\int_{\mathcal{M}^2}|w(y)-w(z)|^{p(y, z)} K(y,z) d v_{{g}}(y) d v_{{g}}(z)\cdot
 \end{equation}
As a consequence of  the Brezis-Lieb Lemma \cite{brezis}, \eqref{k3},  and \eqref{k5}, $\displaystyle L$ is of type $\displaystyle (S_+)\cdot$

We show that  $\displaystyle L$ is a homeomorphism. It is easy to see that $\displaystyle L $ is coercive and injective.  Thanks to the  Minty-Browder Theorem \cite[Theorem 26 A]{zeidler},  $\displaystyle L$ is surjective. So, $\displaystyle L$ is a bijection. There exists a map \\ $\displaystyle G:(W_{K}^{q(y), p(y,z)}(\mathcal{U}))^{*}\rightarrow W_{K}^{q(y), p(y,z)}(\mathcal{U})$ such that $\displaystyle G\circ L= id_{W_{K}^{q(y), p(y,z)}(\mathcal{U})}$ and $\displaystyle L\circ G= id_{(W_{K}^{q(y), p(y,z)}(\mathcal{U}))^{*}}.$\\ 
We show that $\displaystyle G $ is continuous. Let $\displaystyle g_{n}, g \in W_{K}^{q(y), p(y,z)}(\mathcal{U})$ be such that $\displaystyle g_{n}\rightarrow  g   \text {  in } W_{K}^{q(y), p(y,z)}(\mathcal{U}).$  Let $\displaystyle t_{n}=G(g_{n}), w=G(g). $ Then $\displaystyle L(w_{n})=g_{n}$ and $\displaystyle L(w)=g.$ Since $\displaystyle \{t_{n}\}_{n\in \mathbb{N}}$ is bounded in $\displaystyle W_{K}^{q(y), p(y,z)}(\mathcal{U})$, we have $\displaystyle t_{n}\rightharpoonup  w $ in $\displaystyle W_{K}^{q(y), p(y,z)}(\mathcal{U})$. It follows that   $$ \lim _{n \rightarrow+\infty}< L\left(t_{n}\right)-L(w), t_{n}-w>=   \lim _{n \rightarrow+\infty}<g_{n}, t_{n}-w>=0.$$  

Since $\displaystyle L$ is of type $\displaystyle (S_{+})$,  we get $\displaystyle t_{n} \rightarrow w $ in $\displaystyle W_{K}^{q(y), p(y,z)}(\mathcal{U}).$ This completes the proof. 
 \end{proof}
 \begin{lem}\label{l2}
If  $\displaystyle f$ satisfies  $(\mathcal{B}_{1}),$  then  the  operator $\displaystyle S: W_{K}^{q(y), p(y,z)}(\mathcal{U}) \to (\displaystyle W_{K}^{q(y), p(y,z)}(\mathcal{U}))^{*}$ defined by $$\left\langle Sw, \varphi\right\rangle =-\lambda \int_{\mathcal{M}} \beta(y)|w(y)|^{r(y)-2} w \varphi d v_{{g}}(y)-\int_{\mathcal{M}} f(y, w(y)) \varphi d v_{{g}}(y),$$

for every $\displaystyle \varphi \in \left( \displaystyle W_{K}^{q(y), p(y,z)}(\mathcal{U})\right) ^{*} $
 is compact.
 \end{lem}
\begin{proof}
Let 
\begin{align*}
S_{1}&:W_{K}^{q(y), p(y,z)}(\mathcal{U}) \rightarrow L^{q^{'}(y)}(\mathcal{U})&S_{2}&:W_{K}^{q(y), p(y,z)}(\mathcal{U}) \rightarrow L^{q^{'}(y)}(\mathcal{U}) \\
w& \longmapsto   S_{1}w=-\lambda  \beta(y)|w(y)|^{r(y)-2} w(y) &w& \longmapsto   S_{2}w=-f(y,w(y)).  
\end{align*}

We shall show that $\displaystyle  S_{1}$ and $\displaystyle S_{2}$ are both bounded and continuous.
For every $w\in W_{K}^{q(y), p(y,z)}(\mathcal{U}),$
\begin{align*}
\displaystyle |S_{1} w|_{q^{'}(y)}&=\lambda  \int_{\mathcal{M}} |\beta(y)|w(y)|^{r(y)-2} w (y)|^{q^{'}(y)}d v_{{g}}(y)\\
&\leq \lambda ||\beta||_{\infty}\int_{\mathcal{M}} ||w(y)|^{r(y)-1}  |^{q^{'}(y)}d v_{{g}}(y)\\
&\leq \lambda  C ||\beta||_{\infty} \int_{\mathcal{M}} ||w(y)|^{q(y)-1} |^{q^{'}(y}d v_{{g}}(y)\\
&\leq \lambda  C ||\beta||_{\infty} \int_{\mathcal{M}} |w(y)|^{q(y)}d v_{{g}}(y). 
\end{align*}

This implies 
that $\displaystyle S_{1} $ is bounded in $\displaystyle W_{K}^{q(y), p(y,z)}(\mathcal{U}).$
By condition $( \mathcal{B}_{1})$, there exists $\alpha>0 $ such that $$|f(y,w(y))|\leq \alpha(1+|w(y)|^{q(y)-1}).$$

Therefore \begin{align*}
\displaystyle |S_{2} w|_{q^{'}(y)}^{q^{'}(y)}&=\int_{\mathcal{M}} |f(y, w(y)|^{q^{'}(y)} d v_{{g}}(y)\\
&\leq \int_{\mathcal{M}} \alpha(1+|w(y)|^{q(y)-1})^{q^{'}(y)}d v_{{g}}(y)\\
&\leq 2^{q^{^{'}+}}(|\mathcal{M}|+\int_{\mathcal{M}} |w(y)|^{(q(y)-1)q^{'}(y)}d v_{{g}}(y)\\
 &\leq\alpha c^{'}(\mathcal{M}, q(y))\int_{\mathcal{M}} |w(y)|^{(q(y)-1)q^{'}(y)}d v_{{g}}(y)\\
&\leq \alpha c^{'}(\mathcal{U}, q(y)) |w(y)|_{L^{q^{'}(y)}},
\end{align*}

Hence $\displaystyle S_{2}$ is bounded in $\displaystyle W_{K}^{q(y), p(y,z)}(\mathcal{U}).$  Next, we show that 
$\displaystyle S_{2}$ is continuous. Let $ w_{n} \in W_{K}^{q(y), p(y,z)}(\mathcal{U}) $ such that $\displaystyle w_{n} \rightarrow w $ in $\displaystyle W_{K}^{q(y), p(y,z)}(\mathcal{U}).$ 
Then $\displaystyle w_{n} \rightarrow w $ in $\displaystyle L^{q(y)}(\mathcal{U}).$  Hence there exists a subsequence, still denoted by $\displaystyle w_{n},$ and a measurable function $\displaystyle g$ in $\displaystyle L^{q(y)}(\mathcal{U})$ such that $\displaystyle w_{n}(y) \rightarrow w$ and $\displaystyle |w_{n}(y)|\leq  g(y),$ \ a.e in $\mathcal{U}.$ 

Since $f$ is a Carathéodory function, we have 
 \begin{equation}\label{k8}
f(y, w_{n}) \rightarrow  f(y, w(y)) \, \text    {   a.e.  in   }\,  \mathcal{U}.
\end{equation}

According to condition  $ (\mathcal{B}_{1})$, we have that 
\begin{align*}
|f(y, w_{n}(y)|\leq  \alpha (1+ g(y))^{q(y)-1}\in L ^{q^{'}(y)}(\mathcal{U}).
\end{align*}
Using \eqref{k8}, we obtain
$$\int_{\mathcal{M}} |f(y, w_{n}(y))-f(y, w(y))|^{q^{'}(y)} d v_{{g}}(y) \rightarrow 0   \ \text { as }n\rightarrow \infty.$$
The Dominated Convergence Theorem implies that $S_{2} w_{n} \rightarrow  S_{2} w  \text { in }  L^{q^{'}(y)}(\mathcal{U}),$
so $\displaystyle S_{2} $ is continuous in $\displaystyle W_{K}^{q(y), p(y,z)}(\mathcal{U}).$ Because the canonical embedding  $\displaystyle i:W_{K}^{q(y), p(y,z)}(\mathcal{U}) \hookrightarrow L ^{q(y)}(\mathcal{U})$ is compact, its adjoint operator $\displaystyle i^{*}:L ^{q^{'}(y)}(\mathcal{U}) \rightarrow  (W_{K}^{q(y), p(y,z)}(\mathcal{U}))^{*}$ is also compact. As a result,  compositions $\displaystyle i^{*}\circ S_{2}$  and $\displaystyle S_{2} \circ  i^{*}$ are compact, so we come to the conclusion that the operator $\displaystyle S$ is compact and this completes the proof.
\end{proof} 

\hfill

\subsection*{Acknowledgements}
 The authors thank the referees for comments and suggestions.
Repov\v{s} was supported by the Slovenian Research and Innovation Agency grants P1-0292, J1-4031, J1-4001, N1-0278, N1-0114, and N1-0083.

\vfill \eject

 \bibliographystyle{spmpsci}
 \bibliography{myBibLib}

\end{document}